\documentclass[12pt,a4paper]{amsart}
\usepackage[utf8]{inputenc}
\usepackage{amsmath}
\usepackage{amsfonts}
\usepackage{amssymb}
\usepackage{comment}
\usepackage[width=18.00cm, height=24.00cm]{geometry}
\author{Matěj Dostál}
\address{Department of Mathematics, Faculty of Electrical Engineering, Czech Technical University
         in Prague, Czech Republic}
\email{dostamat@math.feld.cvut.cz}
\title{B\'{e}nabou's theorem for pseudoadjunctions}

\usepackage{tikz-cd} 
\tikzcdset{row sep/normal=4.5em, column 
sep/normal=4.5em}

\tikzcdset{every label/.append style = {font = \small}}

\usepackage{amsbsy,mathrsfs,latexsym,amssymb,mathbbol}

\usepackage{mathtools}

\newcommand{\tensor}{\otimes}

\newcommand{\Urar}{\rotatebox[origin=c]{45}{$\Rrightarrow$}}

\newcommand{\Llar}{\rotatebox[origin=c]{180}{$\Rrightarrow$}}
\newcommand{\Ldar}{\rotatebox[origin=c]{225}{$\Rrightarrow$}}
\newcommand{\Ddar}{\rotatebox[origin=c]{-90}{$\Rrightarrow$}}

\newcommand{\uRar}{\rotatebox[origin=c]{45}{\, $\Rightarrow$}}

\newcommand{\dDar}{\Downarrow}

\newcommand{\gkat}[1]{\mathbf{#1}}

\newcommand{\kat}[1]{\mathscr{#1}}
\newcommand{\A}{\kat{A}}

\newcommand{\I}{\kat{I}}

\newcommand{\V}{\kat{V}}
\newcommand{\X}{\kat{X}}

\newcommand{\Gray}{\gkat{Gray}}

\renewcommand{\phi}{\varphi}

\newcommand{\eps}{\varepsilon}

\newcommand{\lift}[1]{\widehat{#1}}

\newcommand{\co}{{\mathit{co}}}
\newcommand{\op}{{\mathit{op}}}

\newcommand{\blank}{{-}}

\newcommand{\adj}{\dashv} 

\tikzset{middle/.style={anchor=center}}
\tikzset{shiftarr/.style={
        rounded corners,%
        to path={--([#1]\tikztostart.center)
                     -- ([#1]\tikztotarget.center) 
                     \tikztonodes
                     -- (\tikztotarget)},
}}

\theoremstyle{plain}
\newtheorem{theorem}{Theorem}[section]

\newtheorem{lemma}[theorem]{Lemma}

\theoremstyle{definition}
\newtheorem{definition}[theorem]{Definition}

\newtheorem{remark}[theorem]{Remark}

\pgfdeclarelayer{background}
\pgfsetlayers{background,main}

\definecolor{cof}{RGB}{219,144,71}
\definecolor{pur}{RGB}{186,146,162}
\definecolor{greeo}{RGB}{91,173,69}
\definecolor{greet}{RGB}{52,111,72}

\numberwithin{equation}{section}

\usepackage[colorlinks,pagebackref,citecolor=blue]{hyperref}
\begin{document}

\begin{abstract}
We give a formal account of B\'{e}nabou's theorem for pseudoadjunctions in the context of $\Gray$-categories. We prove that to give a pseudoadjunction $F \adj U: A \to X$ with unit $\eta$ in a $\Gray$-category $\gkat{K}$ is precisely to give an absolute left (Kan) pseudoextension $U$ of $1_X$ along $F$ witnessed by $\eta$.
\end{abstract}

\maketitle

B\'enabou showed in~\cite{benabou} that an adjunction \[
\begin{tikzcd}
\A
\arrow[d, bend left, "U"]
\arrow[d, phantom, "\adj"]
\\
\X
\arrow[u, bend left, "F"]
\end{tikzcd}
\]
(with unit $\eta$) between categories $\A$ and $\X$ can be characterised as an \emph{absolute left Kan extension}
\[
\begin{tikzcd}
\X
\arrow[r, "F"]
\arrow[dr, "1", swap, ""{name=beg, middle}]
\arrow[dr, phantom, bend left, "\eta \uRar"{middle}]
&
|[alias=A]|\A
\arrow[d, "U", ""{name=fin, middle}]
\\
&
\X
\end{tikzcd}
\]
\emph{of $1$ along $F$}. In this note we are interested in proving a correspondence very similar to the above: a \emph{pseudoadjunction} $F \adj U: A \to X$ (below left)
\[
\begin{tikzcd}
A
\arrow[d, bend left, "U"]
\arrow[d, phantom, "\adj"]
\\
X
\arrow[u, bend left, "F"]
\end{tikzcd}
\qquad
\begin{tikzcd}
X
\arrow[r, "F"]
\arrow[dr, "1", swap, ""{name=beg, middle}]
\arrow[dr, phantom, bend left, "\eta \uRar"{middle}]
&
|[alias=A]|A
\arrow[d, "U", ""{name=fin, middle}]
\\
&
X
\end{tikzcd}
\]
\emph{in a $\Gray$-category $\gkat{K}$} is precisely an \emph{absolute left pseudoexension} (above right) \emph{of $1$ along $F$} in $\gkat{K}$.

Thus our aim is to reproduce B\'{e}nabou's result for the \emph{weaker} notion of a pseudoadjunction. Pseudoadjunctions are interesting: they abound, e.g., in the study of pseudomonads. An important example is the theory of free cocompletions of categories~\cite{kelly:book}. The notion of a pseudoextension already appears in~\cite{nunes-pseudo-kan}.

Instead of working with 2-categories, pseudofunctors and pseudonatural transformations and studying pseudoadjunctions in this setting, we work in the framework of $\Gray$-categories, that is, categories enriched in the category $\V = \Gray$ of 2-categories and 2-functors, equipped with $\Gray$-tensor product~\cite{gray-adjointness-2-categories}.

We will structure this note as follows:
\begin{itemize}
\item The necessary background is covered in Section~\ref{sec:background}.
\item Section~\ref{sec:pseudo-notions-in-gray-cats} contains the definitions of pseudoadjunctions and pseudoextensions.
\item The proof of B\'{e}nabou's theorem appears in Section~\ref{sec:benabou}.
\end{itemize} 

\section{Background}
\label{sec:background}

In this note we shall work with $\Gray$-categories, i.e., categories enriched in the category $\V = \Gray$ (see~\cite{kelly:book} for the theory of enriched categories), where $\Gray$ is a symmetric monoidal closed category of 2-categories with the $\Gray$ tensor product (as described in~\cite{gordon+power+street} or~\cite{gurski}).

\subsection*{$\Gray$-categories}

Recall that a $\Gray$-category $\gkat{K}$ with objects $A$, $B$, $C$, has a hom-2-category $\gkat{K}(A,B)$ for each pair $A$, $B$ of objects, the unit 2-functor $u_{A}: \I \to \gkat{K}(A,A)$ sends the unique $i$-cell ($i = 0,1,2$) to the identity $(i+1)$-cell of $A$ (e.g., the unique object $*$ of $\I$ gets sent to $1_{A}: A \to A$), and the \emph{composition}
\[
\gkat{K}(B,C) \tensor \gkat{K}(A,B) \to \gkat{K}(A,C)
\]
is essentially the composition \emph{cubical} functor
\[
\gkat{K}(B,C) \times \gkat{K}(A,B) \to \gkat{K}(A,C)
\]
yielding, for any $F$ in $\gkat{K}(A,B)$ and any $G$ 
in $\gkat{K}(B,C)$, two 2-functors
\begin{align*}
(\blank) F: \gkat{K}(B,C) & \to 
\gkat{K}(A,C) 
\\
G (\blank): \gkat{K}(A,B) & \to 
\gkat{K}(A,C)
\end{align*}
(``precomposition'' and ``postcomposition''), with $(\blank) F$ acting on the data
\[
\begin{tikzcd}
B
\arrow[rr, bend left, "H"]
\arrow[rr, bend right, "K", swap]
\arrow[phantom, rr, bend left=15, "\tau"{middle}]
\arrow[phantom, rr, "\alpha \Downarrow \, \Rrightarrow \, \Downarrow \beta"{middle}]
&
&
C
\end{tikzcd}
\]
to give
\[
\begin{tikzcd}
B
\arrow[rr, bend left, "HF"]
\arrow[rr, bend right, "KF", swap]
\arrow[phantom, rr, bend left=15, "\tau F"{middle}]
\arrow[phantom, rr, "\alpha F \Downarrow \, \Rrightarrow \, \Downarrow \beta F"{middle}]
&
&
C
\end{tikzcd}
\]
and $G (\blank)$
acting on
\[
\begin{tikzcd}
A
\arrow[rr, bend left, "H"]
\arrow[rr, bend right, "K", swap]
\arrow[phantom, rr, bend left=15, "\tau"{middle}]
\arrow[phantom, rr, "\alpha \Downarrow \, \Rrightarrow \, \Downarrow \beta"{middle}]
&
&
B
\end{tikzcd}
\]
to give
\[
\begin{tikzcd}
A
\arrow[rr, bend left, "GH"]
\arrow[rr, bend right, "GK", swap]
\arrow[phantom, rr, bend left=15, "G \tau"{middle}]
\arrow[phantom, rr, "G \alpha \Downarrow \, \Rrightarrow \, \Downarrow G \beta"{middle}]
&
&
B.
\end{tikzcd}
\]
The 2-functors $(\blank) F$ and $G (\blank)$ are subject to the equality
\[
(G) F = G (F) = GF
\]
and for every pair
\[
\begin{tikzcd}
A
\arrow[r, bend left, "F"]
\arrow[r, bend right, "F'", swap]
\arrow[phantom, r, "\alpha \dDar"{middle}]
&
B
\end{tikzcd}
\qquad
\begin{tikzcd}
B
\arrow[r, shift left, bend left, "G"]
\arrow[r, bend right, "G'", swap]
\arrow[phantom, r, "\beta \dDar"{middle}]
&
C
\end{tikzcd}
\]
there is an isomorphism
\[
\begin{tikzcd}
G F
\arrow[Rightarrow, r, "\beta F"]
\arrow[Rightarrow, d, "G \alpha", swap]
&
G' F
\arrow[phantom, dl, "\Rrightarrow"{middle, 
rotate=225}, 
"\beta_\alpha"{above left}]
\arrow[Rightarrow, d, "G' \alpha"]
\\
G F'
\arrow[Rightarrow, r, "\beta F'", swap]
&
G' F'
\end{tikzcd}
\]
subject to the cubical functor axioms:
\begin{enumerate}
\item Composition axioms
\[
\begin{tikzcd}[arrows=Rightarrow]
GF
\arrow[r, "\beta F"]
\arrow[d, "G \alpha", swap]
&
G'F
\arrow[d, "G' \alpha"]
\arrow[phantom, dl, "\Ldar"{middle}, "\beta_\alpha"{below right}]
\\
GF'
\arrow[r, "\beta F'", swap]
\arrow[d, "G \gamma", swap]
&
G'F'
\arrow[d, "G' \gamma"]
\arrow[phantom, dl, "\Ldar"{middle}, "\beta_\gamma"{below right}]
\\
GF''
\arrow[r, "\beta F''", swap]
&
G'F''
\end{tikzcd}
\quad
=
\quad
\begin{tikzcd}[arrows=Rightarrow]
GF
\arrow[r, "\beta F"]
\arrow[d, "G \alpha", swap]
&
G'F
\arrow[d, "G' \alpha"]
\arrow[phantom, ddl, "\Ldar"{middle}, "\beta_{\gamma \cdot \alpha}"{below right}]
\\
GF'
\arrow[d, "G \gamma", swap]
&
G'F'
\arrow[d, "G' \gamma"]
\\
GF''
\arrow[r, "\beta F''", swap]
&
G'F''
\end{tikzcd}
\]
and
\[
\begin{tikzcd}[arrows=Rightarrow]
GF
\arrow[r, "\beta F"]
\arrow[d, "G \alpha", swap]
&
G'F
\arrow[r, "\delta F"]
\arrow[d, "G' \alpha"]
\arrow[phantom, dl, "\Ldar"{middle}, 
"\beta_\alpha"{below right}]
&
G''F
\arrow[d, "G'' \alpha"]
\arrow[phantom, dl, "\Ldar"{middle}, "\delta_\alpha"{below right}]
\\
GF'
\arrow[r, "\beta F'", swap]
&
G'F'
\arrow[r, "\delta F'", swap]
&
G''F'
\\
&
{=}
&
\end{tikzcd}
\]
\[
\begin{tikzcd}[arrows=Rightarrow]
GF
\arrow[r, "\beta F"]
\arrow[d, "G \alpha", swap]
&
G'F
\arrow[r, "\delta F"]
&
G''F
\arrow[d, "G'' \alpha"]
\arrow[phantom, dll, "\Ldar"{middle}, "{(\delta \cdot \beta)}_\alpha"{below right}]
\\
GF'
\arrow[r, "\beta F'", swap]
&
G'F'
\arrow[r, "\delta F'", swap]
&
G''F'.
\end{tikzcd}
\]
\item ``Modification'' axioms
\[
\begin{tikzcd}
G F
\arrow[Rightarrow, r, "\beta F"]
\arrow[Rightarrow, d, "G \alpha", swap]
&
G' F
\arrow[phantom, dl, "\Rrightarrow"{middle, rotate=225}, "\beta_\alpha"{above left}]
\arrow[Rightarrow, d, "G' \alpha", swap, ""{name=A, middle}]
\arrow[Rightarrow, d, shiftarr={xshift=10ex}, "G' \alpha'", ""{name=B, middle}]
\arrow[from=A, to=B, phantom, "\Llar"{below}, "G' s"{above}]
\\
G F'
\arrow[Rightarrow, r, "\beta F'", swap]
&
G' F'
\end{tikzcd}
\qquad
=
\qquad
\begin{tikzcd}
G F
\arrow[Rightarrow, r, "\beta F"]
\arrow[Rightarrow, d, "G \alpha'", ""{name=A, middle}]
\arrow[Rightarrow, d, shiftarr={xshift=-10ex}, "G \alpha", swap, ""{name=B, middle}]
\arrow[from=A, to=B, phantom, "\Llar"{below}, "G s"{above}]
&
G' F
\arrow[phantom, dl, "\Rrightarrow"{middle, rotate=225}, "\beta_{\alpha'}"{above left}]
\arrow[Rightarrow, d, "G' \alpha'"]
\\
G F'
\arrow[Rightarrow, r, "\beta F'", swap]
&
G' F'
\end{tikzcd}
\]
for any 3-cell $s: \alpha' \Rrightarrow \alpha$ and
\[
\begin{tikzcd}[arrows=Rightarrow]
G F
\arrow[Rightarrow, r, "\beta F", swap]
\arrow[Rightarrow, r, shiftarr={yshift=8ex}, "\beta' F", ""{name=B, middle}]
\arrow[r, bend left, phantom, "t F \Ddar"{middle}]
\arrow[Rightarrow, d, "G \alpha", swap]
&
G' F
\arrow[phantom, dl, "\Rrightarrow"{middle, rotate=225}, "\beta_\alpha"{above left}]
\arrow[Rightarrow, d, "G' \alpha", ""{name=A, middle}]
\\
G F'
\arrow[Rightarrow, r, "\beta F'", swap]
&
G' F'
\end{tikzcd}
\qquad
=
\qquad
\begin{tikzcd}[arrows=Rightarrow]
G F
\arrow[Rightarrow, r, "\beta' F"]
\arrow[Rightarrow, d, "G \alpha", swap]
&
G' F
\arrow[phantom, dl, "\Rrightarrow"{middle, rotate=225}, "{\beta'}_\alpha"{above left}]
\arrow[Rightarrow, d, "G' \alpha", ""{name=A, middle}]
\\
G F'
\arrow[Rightarrow, r, "\beta' F'"]
\arrow[Rightarrow, r, shiftarr={yshift=-8ex}, "\beta F", swap, ""{name=B, middle}]
\arrow[r, bend right, phantom, "t F' \Ddar"{middle}]
&
G' F'
\end{tikzcd}
\]
for any 3-cell $t: \beta' \Rrightarrow \beta$.
\end{enumerate}
Given any triple
\[
\begin{tikzcd}
A
\arrow[r, bend left, "F"]
\arrow[r, bend right, "F'", swap]
\arrow[phantom, r, "\alpha \dDar"{middle}]
&
B
\end{tikzcd}
\qquad
\begin{tikzcd}
B
\arrow[r, shift left, bend left, "G"]
\arrow[r, bend right, "G'", swap]
\arrow[phantom, r, "\beta \dDar"{middle}]
&
C
\end{tikzcd}
\qquad
\begin{tikzcd}
C
\arrow[r, shift left, bend left, "H"]
\arrow[r, bend right, "H'", swap]
\arrow[phantom, r, "\gamma \dDar"{middle}]
&
D
\end{tikzcd}
\]
the associativity equalities
\[
\gamma_{(\beta F)} = (\gamma_\beta) F, \quad 
\gamma_{(G 
\alpha)} = {(\gamma G)}_\alpha, \quad H(\beta_\alpha) 
= 
{(H\beta)}_\alpha
\]
hold, allowing us to relax the notation when working  
with the invertible 3-cells $\beta_\alpha$. 
Finally, the unit equalities
\[
1 F = F, \quad G 1 = G
\]
hold, where $1$ can stand for the identity $1$-cell, 
$2$-cell or $3$-cell on $B$.

\subsection*{Duality}

B\'{e}nabou's theorem admits variations based on duality, both in the case of ordinary categories and in the setting of $\Gray$-categories. We introduce two dual constructions on a $\Gray$-category $\gkat{K}$.

\begin{itemize}
\item The \emph{horizontal dual} $\gkat{K}^\op$ of $\gkat{K}$ is defined by reversing the 1-cells of $\gkat{K}$. That is,
\[
\gkat{K}^\op(A,B) = \gkat{K}(B,A).
\]
Composition in $\gkat{K}^\op$ is defined by the symmetry of the $\Gray$-tensor product, as is usual in the context of enriched categories.

\item The \emph{vertical dual} $\gkat{K}^\co$ of $\gkat{K}$ is defined by reversing the 2-cells of $\gkat{K}$. That is, we put
\[
\gkat{K}^\co(A,B) = (\gkat{K}(A,B))^\op;
\]
observe that the 1-cells of $\gkat{K}$ are not reversed. In this definition we use that $(\gkat{K}(A,B))$ is a \emph{2-category} and that we can therefore form \emph{its} opposite.
\end{itemize}

\section{Pseudoadjunctions and pseudoextensions}
\label{sec:pseudo-notions-in-gray-cats}

We will recall the notion of a pseudoadjunction and a pseudolifting \emph{in} a general $\Gray$-category.

\begin{definition}[Pseudoadjunctions in 
$\Gray$-categories]
\label{def:psadj-in-gray}
Let $\gkat{K}$ be a $\Gray$-category.
We say that 1-cells $U: A \to X$, $F: X \to A$ together with the data
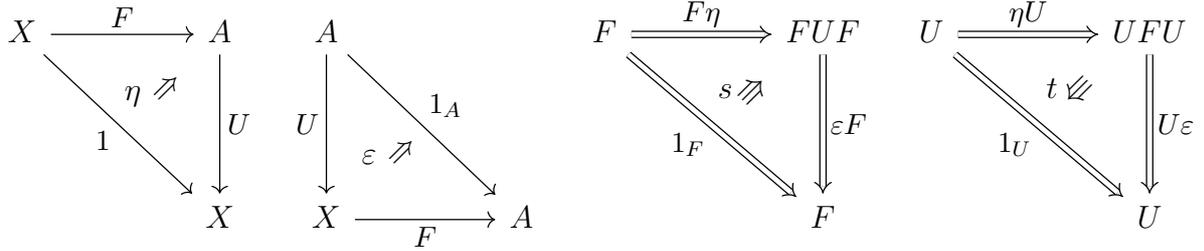
\[
\begin{tikzcd}
X
\arrow[r, "F"]
\arrow[dr, "1", swap, ""{name=beg, middle}]
\arrow[dr, bend left, phantom, "\eta \uRar"{middle}]
&
|[alias=A]|A
\arrow[d, "U", ""{name=fin, middle}]
\\
&
X
\end{tikzcd}
\quad
\begin{tikzcd}
A
\arrow[d, "U", swap]
\arrow[dr, "1_A", ""{name=fin, middle}]
\arrow[dr, bend right, phantom, "\eps \uRar"{middle}]
&
\\
|[alias=X]|X
\arrow[r, "F", swap]
&
A
\end{tikzcd}
\quad
\begin{tikzcd}
F
\arrow[Rightarrow, r, "F \eta"]
\arrow[Rightarrow, dr, "1_F", swap, ""{name=fin, 
middle}]
\arrow[dr, bend left, phantom, "s \Urar"{middle}]
&
|[alias=FUF]|FUF
\arrow[Rightarrow, d, "\eps F"] 
\\
&
F
\end{tikzcd}
\quad
\begin{tikzcd}
U
\arrow[Rightarrow, r, "\eta U"]
\arrow[Rightarrow, dr, "1_U", swap, ""{name=fin, 
middle}]
\arrow[dr, bend left, phantom, "t \Ldar"{middle}]
&
|[alias=UFU]|UFU
\arrow[Rightarrow, d, "U \eps"] 
\\
&
U
\end{tikzcd}
\]
(with $s$ and $t$ being isomorphisms) constitute a 
\emph{pseudoadjunction} in $\gkat{K}$ with 
\emph{unit} 
$\eta$ 
and \emph{counit} $\eps$ if these data satisfy two 
coherence 
identities: the 3-cell
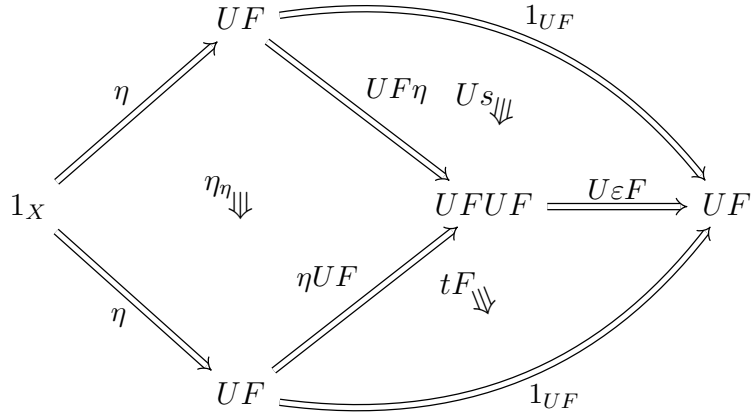
\[
\begin{tikzcd}
&
UF
\arrow[Rightarrow, dr, "UF \eta"]
\arrow[Rightarrow, drr, bend left, "1_{UF}", 
""{name=upp, 
middle}]
\arrow[phantom, dd, "\Rrightarrow"{middle, 
rotate=-90}, "{\eta_\eta}"{above left}]
& & \\
1_X
\arrow[Rightarrow, ur, "\eta"]
\arrow[Rightarrow, dr, "\eta", swap]
& &
UFUF
\arrow[Rightarrow, r, "U \eps F"]
\arrow[phantom, "\Rrightarrow"{middle, 
sloped, rotate=180}, "U s"{above left}, 
from=upp]
& UF \\
&
UF
\arrow[Rightarrow, ur, "\eta UF"]
\arrow[Rightarrow, urr, bend right, "1_{UF}", swap, 
""{name=downn, 
middle}]
\arrow[phantom, urr, "\Rrightarrow"{middle, 
sloped, rotate=-90}, "t F"{above left}]
& & 
\end{tikzcd}
\]
has to be equal to the identity 3-cell on $\eta$, and 
the 3-cell
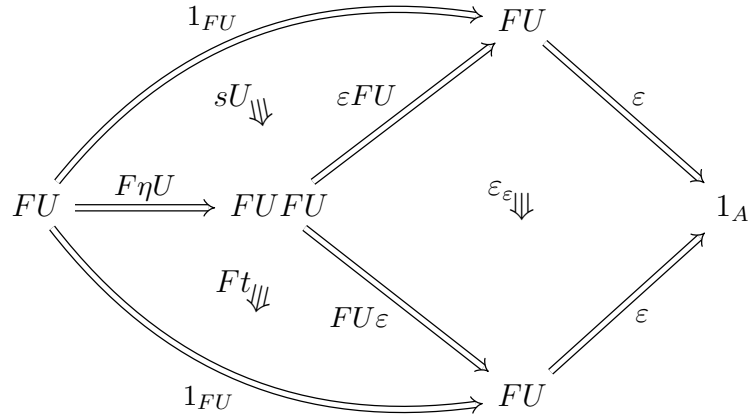
\[
\begin{tikzcd}
& &
FU
\arrow[Rightarrow, rd, "\eps"]
\arrow[
phantom,
dd,
"\Rrightarrow"{sloped, middle},
"{\eps_\eps}"{above left}
]
& \\
FU
\arrow[Rightarrow, urr, bend left, "1_{FU}", 
""{name=upp, 
middle}]
\arrow[Rightarrow, r, "F \eta U"]
\arrow[Rightarrow, drr, bend right, "1_{FU}", swap, 
""{name=downn, 
middle}]
&
FUFU
\arrow[Rightarrow, ur, "\eps FU"]
\arrow[Rightarrow, dr, "FU \eps", swap]
\arrow[
phantom,
"\Rrightarrow"{sloped, middle},
"s U"{above left},
from=upp
]
\arrow[
phantom,
"\Rrightarrow"{sloped, middle, rotate=180},
"F t"{above left},
to=downn
]
& & 1_A \\
& &
FU
\arrow[Rightarrow, ur, "\eps", swap]
&
\end{tikzcd}
\]
has to be equal to the identity 3-cell on $\eps$. We write $F \adj U: A \to X$ for this pseudoadjunction.
\end{definition}

Duality operations with the $\Gray$-category $\gkat{K}$ transform pseudoadjunctions into pseudoadjunctions. The roles of the defining data have to be swapped accordingly.

\begin{remark}
Suppose we are given the category $\gkat{K}$ as in 
Definition~\ref{def:psadj-in-gray} and the data for 
the pseudoadjunction $F \adj U: A \to X$.
\begin{enumerate}
\item In $\gkat{K}^\op$, the same data transform into 
a pseudoadjunction $U \adj F: X \to A$ due to the 
reversal of 1-cells. The unit $\eta$ and counit 
$\eps$ 
stay the same, as well as the coherence 3-cells $s$ 
and $t$.
\item In $\gkat{K}^\co$, the same data transform into 
a pseudoadjunction $U \adj F: X \to A$, but with unit 
$\eps$ and counit $\eta$; the coherence 3-cells $s$ 
and $t$ stay the same, although their role as 
witnesses for the triangle isomorphisms is swapped.
\end{enumerate}
\end{remark}

The notion of a (left) pseudoextension is the appropriate weakening of the usual notion of a (left) Kan extension.

\begin{definition}[Left pseudoextension~\cite{fiore:pseudo-limits,nunes-pseudo-kan}]
\label{def:left-pseudoextension}
In a $\Gray$-category $\gkat{K}$, we say that
\[
\begin{tikzcd}
X
\arrow[r, "J"]
\arrow[dr, "H", swap, ""{name=beg, middle}]
\arrow[dr, phantom, bend left, "\eta \uRar"{middle}]
&
|[alias=A]|A
\arrow[d, "L", ""{name=fin, middle}]
\\
&
B
\end{tikzcd}
\]
exhibits $L$ as a left pseudoextension of $H$ along 
$J$ if for each
\[
\begin{tikzcd}
X
\arrow[r, "J"]
\arrow[dr, "H", swap, ""{name=beg, middle}]
\arrow[dr, phantom, bend left, "f \uRar"{middle}]
&
|[alias=A]|A
\arrow[d, "K", ""{name=fin, middle}]
\\
&
B
\end{tikzcd}
\]
(i.e., $f: H \Rightarrow KJ$) there is a 2-cell 
$f^\sharp: L 
\Rightarrow K$ and an isomorphism 3-cell
\[
\begin{tikzcd}
H
\arrow[Rightarrow, r, "\eta"]
\arrow[Rightarrow, dr, "f", swap, ""{name=fin, middle}]
\arrow[dr, phantom, bend left, "\Ldar \mu(f)"{middle}]
&
|[alias=LJ]|LJ
\arrow[Rightarrow, d, "f^\sharp J"] 
\\
&
KJ
\end{tikzcd}
\]
such that for each $k: L \Rightarrow K$ and a 3-cell
\[
\begin{tikzcd}
H
\arrow[Rightarrow, r, "\eta"]
\arrow[Rightarrow, dr, "f", swap, ""{name=fin, 
middle}]
\arrow[dr, phantom, bend left, "\Ldar \omega"{middle}]
&
|[alias=LJ]|LJ
\arrow[Rightarrow, d, "kJ"] 
\\
&
KJ
\end{tikzcd}
\]
there is a unique 3-cell $\lift{\omega}: k 
\Rrightarrow f^\sharp$ such that
\[
\begin{tikzcd}
H
\arrow[Rightarrow, r, "\eta"]
\arrow[Rightarrow, dr, "f", swap, ""{name=fin1, 
middle}]
&
|[alias=LJ]|LJ
\arrow[Rightarrow, d, "f^\sharp J", swap, ""{name=fin2, middle}]
\\
&
|[alias=KJ]|KJ
\arrow[phantom, from=LJ, to=fin1, "\Rrightarrow"{near 
end, rotate=225}, 
"\mu(f)"{near 
start}]
\arrow[Rightarrow, from=LJ, to=KJ, 
shiftarr={xshift=10ex}, 
"kJ", ""{name=beg2, middle}]
\arrow[from=beg2, to=fin2, "\Rrightarrow"{below, 
rotate=180}, 
"\lift{\omega} J"{below}, phantom]
\end{tikzcd}
=
\begin{tikzcd}
H
\arrow[Rightarrow, r, "\eta"]
\arrow[Rightarrow, dr, "f", swap, ""{name=fin, middle}]
\arrow[dr, phantom, bend left, "\Ldar \omega"{middle}]
&
|[alias=LJ]|LJ
\arrow[Rightarrow, d, "kJ"]
\\
&
KJ
\end{tikzcd}
\]

We say that the pseudoextension $\eta: H \to LJ$ is 
\emph{preserved by $G: B \to C$} if the 2-cell
\[
\begin{tikzcd}
X
\arrow[r, "J"]
\arrow[dr, "GH", swap, ""{name=beg, middle}]
\arrow[dr, phantom, bend left, "G \eta \uRar"{middle}]
&
|[alias=A]|A
\arrow[d, "GL", ""{name=fin, middle}]
\\
&
C
\end{tikzcd}
\]
exhibits $GL$ as a left pseudoextension of $GH$ along 
$J$.

The pseudoextension $\eta: H \to LJ$ is said to be 
\emph{absolute} if it is preserved by \emph{any} 
1-cell $G: B \to C$.
\end{definition}

The various notions of duality for $\Gray$-categories allow us to express compactly the definition of (left/right) pseudoextensions and pseudoliftings via the definition of a left pseudoextension.

\begin{definition}
Given a left pseudoextension 
\[
\begin{tikzcd}
X
\arrow[r, "J"]
\arrow[dr, "H", swap, ""{name=beg, middle}]
\arrow[dr, phantom, bend left, "\eta \uRar"{middle}]
&
|[alias=A]|A
\arrow[d, "L", ""{name=fin, middle}]
\\
&
B
\end{tikzcd}
\]
in a $\Gray$-category $\gkat{K}$, we call it
\begin{enumerate}
\item a \emph{right} pseudoextension of $H$ along $J$ 
in $\gkat{K}^\co$.
\item a \emph{left} pseudolifting of $H$ through $J$ 
in $\gkat{K}^\op$.
\item a \emph{right} pseudolifting of $H$ through $J$ 
in $\gkat{K}^{\co \op}$.
\end{enumerate}
\end{definition}

With these definitions at hand, we can move to the statement and proof of B\'{e}nabou's theorem.

\section{B\'{e}nabou's theorem for pseudoextensions}
\label{sec:benabou}

Let us first recall ordinary B\'{e}nabou's theorem~\cite{benabou}.

\begin{theorem}
For functors $U: \A \to \X$ and $F: \X \to \A$ the following are equivalent:
\begin{enumerate}
\item $F \adj U$ holds with unit $\eta$.
\item $\eta$ exhibits $U$ as an absolute left extension of $1_\X$ along $F$.
\item $\eta$ exhibits $U$ as a left extension of $1_\X$ along $F$, and this extension is preserved by $F$.
\item $\eta$ exhibits $F$ as an absolute left lifting of $1_\X$ through $U$.
\item $\eta$ exhibits $F$ as a left lifting of $1_\X$ through $U$, and this lifting is preserved by $U$.
\end{enumerate}
\end{theorem}

The pseudo-version of B\'{e}nabou's theorem can be stated in the same way, changing the notions of adjunction and extension/lifting to the notions of pseudoadjunction and pseudoextension/pseudolifting.

\begin{theorem}
\label{thm:pseudo-benabou}
Given two 1-cells $U: A \to X$ and $F: X \to A$ in a $\Gray$-category $\gkat{K}$, the following are equivalent:
\begin{enumerate}
\item $F \adj U$ is a pseudoadjunction with unit $\eta$.
\item $\eta$ exhibits $U$ as an absolute left pseudoextension of $1_X$ along $F$.
\item $\eta$ exhibits $U$ as a left pseudoextension of $1_X$ along $F$, and this extension is preserved by $F$.
\item $\eta$ exhibits $F$ as an absolute left pseudolifting of $1_X$ through $U$.
\item $\eta$ exhibits $F$ as a left pseudolifting of $1_X$ through $U$, and this lifting is preserved by $U$.
\end{enumerate}
\end{theorem}

The proof strategy in the ordinary case and in the pseudo-case is the same: it is enough to prove the implications (1) $\implies$ (2) and (3) $\implies$ (1). This is because (2) $\implies$ (3) is trivial, and because the equivalence of (1), (4) and (5) follows by duality. Moreover, the ordinary proofs can serve as a guidance for the proofs of the pseudo-case.

\begin{lemma}[The implication (3) $\implies$ (1)]
Suppose that
\[
\begin{tikzcd}
X
\arrow[r, "F"]
\arrow[dr, "1", swap, ""{name=beg, middle}]
\arrow[dr, phantom, bend left, "\eta \uRar"{middle}]
&
|[alias=A]|A
\arrow[d, "U", ""{name=fin, middle}]
\\
&
X
\end{tikzcd}
\]
is a left (Kan) pseudoextension preserved by $F$. 
Then 
$\eta$ can be made a unit of a pseudoadjunction $F 
\adj U$.
\end{lemma}
\begin{proof}
Recall that
\[
\begin{tikzcd}
X
\arrow[r, "F"]
\arrow[dr, "1", swap, ""{name=beg, middle}]
\arrow[dr, phantom, bend left, "\eta \uRar"{middle}]
&
|[alias=A]|A
\arrow[d, "U", ""{name=fin, middle}]
\\
&
X
\end{tikzcd}
\]
is a left pseudoextension if for each
\[
\begin{tikzcd}
X
\arrow[r, "F"]
\arrow[dr, "1", swap, ""{name=beg, middle}]
\arrow[dr, phantom, bend left, "f \uRar"{middle}]
&
|[alias=A]|A
\arrow[d, "K", ""{name=fin, middle}]
\\
&
X
\end{tikzcd}
\]
(i.e., $f: 1_X \Rightarrow KF$) there is a 2-cell 
$f^\sharp: U 
\Rightarrow K$ and an isomorphism 3-cell
\[
\begin{tikzcd}
1_X
\arrow[Rightarrow, r, "\eta"]
\arrow[Rightarrow, dr, "f", swap, ""{name=fin, 
middle}]
\arrow[dr, phantom, bend left, "\Ldar \mu(f)"{middle}]
&
|[alias=UF]|UF
\arrow[Rightarrow, d, "f^\sharp F"] 
\\
&
KF
\end{tikzcd}
\]
such that for each $k: U \Rightarrow K$ and a 3-cell
\[
\begin{tikzcd}
1_X
\arrow[Rightarrow, r, "\eta"]
\arrow[Rightarrow, dr, "f", swap, ""{name=fin, 
middle}]
\arrow[dr, phantom, bend left, "\Ldar \omega"{middle}]
&
|[alias=UF]|UF
\arrow[Rightarrow, d, "kF"] 
\\
&
KF
\end{tikzcd}
\]
there is a unique 3-cell $\lift{\omega}: k 
\Rrightarrow f^\sharp$ such that
\[
\begin{tikzcd}
1_X
\arrow[Rightarrow, r, "\eta"]
\arrow[Rightarrow, dr, "f", swap, ""{name=fin1, 
middle}]
&
|[alias=UF]|UF
\arrow[Rightarrow, d, "f^\sharp F", swap, 
""{name=fin2, middle}] 
\\
&
|[alias=KF]|KF
\arrow[phantom, from=UF, to=fin1, "\Rrightarrow"{near 
end, rotate=225}, 
"\mu(f)"{near 
start}]
\arrow[Rightarrow, from=UF, to=KF, 
shiftarr={xshift=10ex}, 
"kF", ""{name=beg2, middle}]
\arrow[from=beg2, to=fin2, "\Rrightarrow"{below, 
rotate=180}, 
"\lift{\omega} F"{below}, phantom]
\end{tikzcd}
=
\begin{tikzcd}
1_X
\arrow[Rightarrow, r, "\eta"]
\arrow[Rightarrow, dr, "f", swap, ""{name=fin, 
middle}]
\arrow[dr, phantom, bend left, "\Ldar \omega"{middle}]
&
|[alias=UF]|UF
\arrow[Rightarrow, d, "kF"] 
\\
&
KF
\end{tikzcd}
\]
For the purpose of establishing notation, we describe 
the data concerning the left pseudoextension
\[
\begin{tikzcd}
X
\arrow[r, "F"]
\arrow[dr, "F", swap, ""{name=beg, middle}]
\arrow[dr, phantom, bend left, "F \eta \uRar"{middle}]
&
|[alias=A]|A
\arrow[d, "FU", ""{name=fin, middle}]
\\
&
A
\end{tikzcd}
\]
Given, e.g., the identity
\[
\begin{tikzcd}
X
\arrow[r, "F"]
\arrow[dr, "F", swap, ""{name=beg, middle}]
\arrow[dr, phantom, bend left, "1_F \uRar"{middle}]
&
|[alias=A]|A
\arrow[d, "1_A", ""{name=fin, middle}]
\\
&
A
\end{tikzcd}
\]
(the 2-cell $1_F: F \Rightarrow F$), we have a 2-cell
$(1_A)^\sharp: FU \Rightarrow 1_A$ that we will 
denote 
by $\eps$ and which will be the counit of the 
pseudoadjunction we construct. With this counit comes 
an isomorphism
\[
\begin{tikzcd}
F
\arrow[Rightarrow, r, "F \eta"]
\arrow[Rightarrow, dr, "1_F", swap, ""{name=fin, 
middle}]
\arrow[dr, phantom, bend left, "\Ldar s^{-1}"{middle}]
&
|[alias=FUF]|FUF
\arrow[Rightarrow, d, "\eps F"] 
\\
&
F
\end{tikzcd}
\]
such that for each $h: FU \Rightarrow 1_A$ and a 
3-cell
\[
\begin{tikzcd}
F
\arrow[Rightarrow, r, "F \eta"]
\arrow[Rightarrow, dr, "1_F", swap, ""{name=fin, 
middle}]
\arrow[dr, phantom, bend left, "\Ldar \nu"{middle}]
&
|[alias=FUF]|FUF
\arrow[Rightarrow, d, "h F"] 
\\
&
F
\end{tikzcd}
\]
there is a unique $\lift{\nu}: h \Rrightarrow \eps$ 
satisfying
\[
\begin{tikzcd}
F
\arrow[Rightarrow, r, "F \eta"]
\arrow[Rightarrow, dr, "1_F", swap, ""{name=fin1, 
middle}]
&
|[alias=FUF]|FUF
\arrow[Rightarrow, d, "\eps F", swap, 
""{name=fin2, middle}] 
\\
&
|[alias=F]|F
\arrow[phantom, from=FUF, to=fin1, 
"\Rrightarrow"{near 
end, rotate=225}, 
"s^{-1}"{near 
start}]
\arrow[Rightarrow, from=FUF, to=F, 
shiftarr={xshift=10ex}, 
"kF", ""{name=beg2, middle}]
\arrow[from=beg2, to=fin2, "\Rrightarrow"{below, 
rotate=180}, 
"\lift{\nu} F"{below}, phantom]
\end{tikzcd}
=
\begin{tikzcd}
F
\arrow[Rightarrow, r, "F \eta"]
\arrow[Rightarrow, dr, "1_F", swap, ""{name=fin, middle}]
\arrow[dr, phantom, bend left, "\Ldar \nu"{middle}]
&
|[alias=FUF]|UF
\arrow[Rightarrow, d, "h F"] 
\\
&
F
\end{tikzcd}
\]
Let us first observe that $s^{-1}$ (or, equivalently, 
its inverse $s$) witnesses the first 
triangle axiom of a pseudomonad. To obtain the second 
triangle isomorphism, consider that $\eta: 1_X 
\Rightarrow UF$ lifts to the identity $1_U = 
\eta^\sharp: U \Rightarrow U$ with the identity 
3-cell 
($\mu(\eta) = 1_\eta$)
\[
\begin{tikzcd}
1_X
\arrow[Rightarrow, r, "\eta"]
\arrow[Rightarrow, dr, "\eta", swap, ""{name=fin, 
middle}]
\arrow[dr, phantom, bend left, "\Ldar \mu(\eta)"{middle}]
&
|[alias=UF]|UF
\arrow[Rightarrow, d, "1_{UF}"] 
\\
&
UF
\end{tikzcd}
\]
By the universal property of the left pseudoextension 
given by $\eta$ we get that for the 2-cell $$U 
\xRightarrow{\eta U} UFU \xRightarrow{U \eps} U$$ and 
the 3-cell
\[
\begin{tikzcd}
1_X
\arrow[Rightarrow, r, "\eta"]
\arrow[Rightarrow, d, "\eta", swap]
&
UF
\arrow[Rightarrow, d, "\eta UF"]
\arrow[dl, phantom, "\Rrightarrow"{rotate=225}, 
"{\eta_\eta}^{-1}"{near start}]
\\
UF
\arrow[Rightarrow, r, "UF \eta"]
\arrow[Rightarrow, dr, "1_{UF}", swap, ""{name=fin, 
middle}]
\arrow[dr, phantom, bend left, "\Ldar Us^{-1}"{middle}]
&
|[alias=UFUF]|UFUF
\arrow[Rightarrow, d, "U \eps F"] 
\\
&
UF
\end{tikzcd}
\]
there is a unique 3-cell
\[
\begin{tikzcd}
U
\arrow[Rightarrow, r, "\eta U"]
\arrow[Rightarrow, dr, "1_U", swap, ""{name=fin, 
middle}]
\arrow[dr, bend left, phantom, "t \Ldar"{middle}]
&
|[alias=UFU]|UFU
\arrow[Rightarrow, d, "U \eps"] 
\\
&
U
\end{tikzcd}
\]
such that the 3-cell
\[
\begin{tikzcd}
1_X
\arrow[Rightarrow, r, "\eta"]
&
UF
\arrow[Rightarrow, r, "\eta UF"]
\arrow[Rightarrow, dr, "1_{UF}", swap, ""{name=fin, 
middle}]
\arrow[dr, bend left, phantom, "tF \Ldar"{middle}]
&
|[alias=UFUF]|UFUF
\arrow[Rightarrow, d, "U \eps F"] 
\\
&&
UF
\end{tikzcd}
\]
equals
\[
\begin{tikzcd}
1_X
\arrow[Rightarrow, r, "\eta"]
\arrow[Rightarrow, d, "\eta", swap]
&
UF
\arrow[Rightarrow, d, "\eta UF"]
\arrow[dl, phantom, "\Rrightarrow"{rotate=225}, 
"{\eta_\eta}^{-1}"{near start}]
\\
UF
\arrow[Rightarrow, r, "UF \eta"]
\arrow[Rightarrow, dr, "1_{UF}", swap, ""{name=fin, 
middle}]
\arrow[dr, bend left, phantom, "Us^{-1} \Ldar"{middle}]
&
|[alias=UFUF]|UFUF
\arrow[Rightarrow, d, "U \eps F"] 
\\
&
UF
\end{tikzcd}
\]
or, written differently, that the 3-cell
\[
\begin{tikzcd}
&
UF
\arrow[Rightarrow, dr, "UF \eta"]
\arrow[Rightarrow, drr, bend left, "1_{UF}", 
""{name=upp, 
middle}]
\arrow[phantom, dd, "\Rrightarrow"{middle, 
rotate=-90}, "{\eta_\eta}"{above left}]
& & \\
1_X
\arrow[Rightarrow, ur, "\eta"]
\arrow[Rightarrow, dr, "\eta", swap]
& &
UFUF
\arrow[Rightarrow, r, "U \eps F"]
\arrow[phantom, "\Rrightarrow"{middle, 
sloped, rotate=180}, "U s"{above left}, 
from=upp]
& UF \\
&
UF
\arrow[Rightarrow, ur, "\eta UF"]
\arrow[Rightarrow, urr, bend right, "1_{UF}", swap, 
""{name=downn, 
middle}]
\arrow[phantom, urr, "\Rrightarrow"{middle, 
sloped, rotate=-90}, "t F"{above left}]
& & 
\end{tikzcd}
\]
is equal to identity. This is precisely the first 
coherence axiom for pseudoadjunctions.

For the other coherence axiom, we need the 3-cell
\[
\begin{tikzcd}
& &
FU
\arrow[Rightarrow, rd, "\eps"]
\arrow[
phantom,
dd,
"\Rrightarrow"{sloped, middle},
"{\eps_\eps}"{above left}
]
& \\
FU
\arrow[Rightarrow, urr, bend left, "1_{FU}", 
""{name=upp, 
middle}]
\arrow[Rightarrow, r, "F \eta U"]
\arrow[Rightarrow, drr, bend right, "1_{FU}", swap, 
""{name=downn, 
middle}]
&
FUFU
\arrow[Rightarrow, ur, "\eps FU"]
\arrow[Rightarrow, dr, "FU \eps", swap]
\arrow[
phantom,
"\Rrightarrow"{sloped, middle},
"s U"{above left},
from=upp
]
\arrow[
phantom,
"\Rrightarrow"{sloped, middle, rotate=180},
"F t"{above left},
to=downn
]
& & 1_A \\
& &
FU
\arrow[Rightarrow, ur, "\eps", swap]
&
\end{tikzcd}
\]
to be equal to identity as well.

We shall use that for each $h: FU \Rightarrow 1_A$ 
and 
a 3-cell
\[
\begin{tikzcd}
F
\arrow[Rightarrow, r, "F \eta"]
\arrow[Rightarrow, dr, "1_F", swap, ""{name=fin, 
middle}]
\arrow[dr, bend left, phantom, "\nu \Ldar"{middle}]
&
|[alias=FUF]|FUF
\arrow[Rightarrow, d, "h F"] 
\\
&
F
\end{tikzcd}
\]
there is a \emph{unique} $\lift{\nu}: h \Rrightarrow 
\eps$ 
satisfying
\[
\begin{tikzcd}
F
\arrow[Rightarrow, r, "F \eta"]
\arrow[Rightarrow, dr, "1_F", swap, ""{name=fin1, 
middle}]
\arrow[dr, bend left, phantom, "s^{-1} \Ldar"{middle}]
&
|[alias=FUF]|FUF
\arrow[Rightarrow, d, "\eps F", swap, 
""{name=fin2, middle}] 
\\
&
|[alias=F]|F
\arrow[Rightarrow, from=FUF, to=F, 
shiftarr={xshift=10ex}, 
"kF", ""{name=beg2, middle}]
\arrow[from=beg2, to=fin2, "\Rrightarrow"{below, 
rotate=180}, 
"\lift{\nu} F"{below}, phantom]
\end{tikzcd}
=
\begin{tikzcd}
F
\arrow[Rightarrow, r, "F \eta"]
\arrow[Rightarrow, dr, "1_F", swap, ""{name=fin, 
middle}]
\arrow[dr, bend left, phantom, "\nu \Ldar"{middle}]
&
|[alias=FUF]|UF
\arrow[Rightarrow, d, "h F"] 
\\
&
F
\end{tikzcd}
\]
Thus if we find \emph{two} 3-cells $\alpha, \beta: h 
\Rrightarrow \eps$ with
\[
\begin{tikzcd}
F
\arrow[Rightarrow, r, "F \eta"]
\arrow[Rightarrow, dr, "1_F", swap, ""{name=fin1, 
middle}]
\arrow[dr, bend left, phantom, "s^{-1} \Ldar"{middle}]
&
|[alias=FUF]|FUF
\arrow[Rightarrow, d, "\eps F", swap, 
""{name=fin2, middle}] 
\\
&
|[alias=F]|F
\arrow[Rightarrow, from=FUF, to=F, 
shiftarr={xshift=10ex}, 
"kF", ""{name=beg2, middle}]
\arrow[from=beg2, to=fin2, "\Rrightarrow"{below, 
rotate=180}, 
"\alpha F"{below}, phantom]
\end{tikzcd}
=
\begin{tikzcd}
F
\arrow[Rightarrow, r, "F \eta"]
\arrow[Rightarrow, dr, "1_F", swap, ""{name=fin1, 
middle}]
\arrow[dr, bend left, phantom, "s^{-1} \Ldar"{middle}]
&
|[alias=FUF]|FUF
\arrow[Rightarrow, d, "\eps F", swap, 
""{name=fin2, middle}] 
\\
&
|[alias=F]|F
\arrow[Rightarrow, from=FUF, to=F, 
shiftarr={xshift=10ex}, 
"kF", ""{name=beg2, middle}]
\arrow[from=beg2, to=fin2, "\Rrightarrow"{below, 
rotate=180}, 
"\beta F"{below}, phantom]
\end{tikzcd}
\]
it means that $\alpha = \beta$. Take now the 3-cell
\[
\begin{tikzcd}
FU
\arrow[Rightarrow, r, "F\eta U"]
\arrow[Rightarrow, dr, "1_{FU}", swap, ""{name=fin, 
middle}]
\arrow[dr, bend left, phantom, "Ft \Ldar"{middle}]
&
|[alias=FUFU]|FUFU
\arrow[Rightarrow, d, "FU \eps"] 
&
\\
&
FU
\arrow[Rightarrow, r, "\eps"]
&
1_A
\end{tikzcd}
\]
for $\alpha$ and the 3-cell
\[
\begin{tikzcd}
FU
\arrow[Rightarrow, r, "F \eta U"]
\arrow[Rightarrow, dr, "1_{FU}", swap, ""{name=fin, 
middle}]
\arrow[dr, bend left, phantom, "s^{-1}U \Ldar"{middle}]
&
|[alias=FUFU]|FUFU
\arrow[Rightarrow, d, "\eps FU"]
\arrow[Rightarrow, r, "FU \eps"]
&
FU
\arrow[Rightarrow, d, "\eps"]
\arrow[dl, phantom, "\Rrightarrow"{rotate=225}, 
"{\eps_\eps}^{-1}"{near start}]
\\
&
FU
\arrow[Rightarrow, r, "\eps", swap]
&
1_A
\end{tikzcd}
\]
for $\beta$.
We ask whether the 3-cells
\[
\begin{tikzcd}
F
\arrow[Rightarrow, r, "F \eta"]
\arrow[Rightarrow, rrdd, "1_F", swap, ""{name=fin2, 
middle}]
&
FUF
\arrow[Rightarrow, r, "F \eta UF"]
\arrow[Rightarrow, dr, "1_{FUF}", swap, ""{name=fin, 
middle}]
\arrow[dr, bend left, phantom, "FtF \Ldar"{middle}]
&
|[alias=FUFUF]|FUFUF
\arrow[Rightarrow, d, "FU \eps F"] 
\\
&
&
|[alias=FUF]|FUF
\arrow[Rightarrow, d, "\eps F"]
\\
&
&
F
\arrow[from=FUF, to=fin2, phantom, 
"\Rrightarrow"{rotate=225}, 
"s^{-1}"{near start}]
\end{tikzcd}
\]
and
\[
\begin{tikzcd}
&
FUF
\arrow[Rightarrow, r, "F \eta U F"]
\arrow[Rightarrow, dr, "1_{FUF}", swap, ""{name=fin, 
middle}]
\arrow[dr, bend left, phantom, "s^{-1}UF \Ldar"{middle}]
&
|[alias=FUFUF]|FUFUF
\arrow[Rightarrow, d, "\eps FUF"]
\arrow[Rightarrow, r, "FU \eps F"]
&
FUF
\arrow[Rightarrow, d, "\eps F"]
\arrow[dl, phantom, "\Rrightarrow"{rotate=225}, 
"{\eps_{\eps F}}^{-1}"{near start}]
\\
F
\arrow[Rightarrow, ur, "F \eta"]
\arrow[Rightarrow, rrr, bend right, "1_F", swap]
\arrow[rr, phantom, "\Rrightarrow"{rotate=225, 
above}, 
"s^{-1}"{above}]
&
&
FUF
\arrow[Rightarrow, r, "\eps F", swap]
&
F
\end{tikzcd}
\]
are equal. Pasting $s$ and $F\eta_\eta$, we can 
equivalently ask whether the 3-cells
\begin{equation}
\label{diag:tmp1}
\begin{tikzcd}
F
\arrow[Rightarrow, d, "F \eta", swap]
\arrow[Rightarrow, r, "F \eta"]
&
FUF
\arrow[Rightarrow, d, "FUF \eta"]
\arrow[dl, phantom, "\Rrightarrow"{rotate=225}, "F 
\eta_\eta"{near 
start}]
&
\\
FUF
\arrow[Rightarrow, r, "F \eta UF"]
\arrow[Rightarrow, dr, "1_{FUF}", swap, ""{name=fin, 
middle}]
\arrow[dr, bend left, phantom, "FtF \Ldar"{middle}]
&
|[alias=FUFUF]|FUFUF
\arrow[Rightarrow, d, "FU \eps F"] 
&
\\
&
|[alias=FUF]|FUF
\arrow[Rightarrow, r, "\eps F"]
&
F
\\
&
&
\end{tikzcd}
\end{equation}
and
\begin{equation}
\label{diag:tmp2}
\begin{tikzcd}
F
\arrow[Rightarrow, d, "F \eta", swap]
\arrow[Rightarrow, r, "F \eta"]
&
FUF
\arrow[Rightarrow, d, "FUF \eta"]
\arrow[dl, phantom, "\Rrightarrow"{rotate=225}, 
"F{\eta_{\eta}}"{near start}]
&
\\
FUF
\arrow[Rightarrow, r, "F \eta U F"]
\arrow[Rightarrow, dr, "1_{FUF}", swap, ""{name=fin, 
middle}]
\arrow[dr, bend left, phantom, "s^{-1}UF \Ldar"{middle}]
&
|[alias=FUFUF]|FUFUF
\arrow[Rightarrow, d, "\eps FUF"]
\arrow[Rightarrow, r, "FU \eps F"]
&
FUF
\arrow[Rightarrow, d, "\eps F"]
\arrow[dl, phantom, "\Rrightarrow"{rotate=225}, 
"{\eps_{\eps F}}^{-1}"{near start}]
\\
&
FUF
\arrow[Rightarrow, r, "\eps F", swap]
&
F
\end{tikzcd}
\end{equation}
are equal. Using the first coherence axiom, the 
diagram~\eqref{diag:tmp1} is equal to
\begin{equation}
\label{diag:tmp3}
\begin{tikzcd}
F
\arrow[Rightarrow, r, "F \eta"]
&
FUF
\arrow[Rightarrow, r, "FUF \eta"]
\arrow[Rightarrow, dr, "1_{FUF}", swap, ""{name=fin, 
middle}]
&
|[alias=FUFUF]|FUFUF
\arrow[Rightarrow, d, "FU \eps F"] 
&
\\
&
&
|[alias=FUF]|FUF
\arrow[Rightarrow, r, "\eps F"]
&
F
\\
&
&
&
\arrow[phantom, from=FUFUF, to=fin, 
"\Rrightarrow"{near 
end, rotate=225}, 
"FU s^{-1}"{near 
start}]
\end{tikzcd}
\end{equation}
Let us take diagrams~\eqref{diag:tmp2} and 
\eqref{diag:tmp3} and paste $\eps_{\eps F}$ and 
$\eps_{F\eta}$. The resulting diagrams
\begin{equation}
\label{diag:tmp4}
\begin{tikzcd}
F
\arrow[Rightarrow, d, "F \eta", swap]
\arrow[Rightarrow, r, "F \eta"]
&
FUF
\arrow[Rightarrow, d, "FUF \eta", swap]
\arrow[Rightarrow, dr, "\eps F"]
\arrow[dl, phantom, "\Rrightarrow"{rotate=225}, 
"F{\eta_{\eta}}"{near start}]
&
\\
FUF
\arrow[Rightarrow, r, "F \eta U F"]
\arrow[Rightarrow, dr, "1_{FUF}", swap, ""{name=fin, 
middle}]
&
|[alias=FUFUF]|FUFUF
\arrow[Rightarrow, d, "\eps FUF", swap]
&
F
\arrow[Rightarrow, dl, "F \eta"]
\arrow[phantom, l, "\Rrightarrow"{rotate=180, above}, 
"\eps_{F \eta}"{above}]
\\
&
FUF
\arrow[Rightarrow, r, "\eps F", swap]
&
F
\arrow[phantom, from=FUFUF, to=fin, 
"\Rrightarrow"{near 
end, rotate=225}, 
"s^{-1} UF"{near 
start}]
\end{tikzcd}
\end{equation}
and
\begin{equation}
\begin{tikzcd}
& &
F
\arrow[Rightarrow, r, "F \eta"]
\arrow[phantom, d, "\Rrightarrow"{rotate=-90}, 
"\eps_{F \eta}"{near start}]
&
FUF
\arrow[Rightarrow, dr, "\eps F"]
\arrow[d, phantom, "\Rrightarrow"{rotate=-90}, 
"\eps_{\eps F}"{near start}]
&
\\
F
\arrow[Rightarrow, r, "F \eta"]
&
FUF
\arrow[Rightarrow, r, "F \eta UF"]
\arrow[Rightarrow, ur, "\eps F", 
""{name=beg, 
middle}]
\arrow[Rightarrow, rr, "1_{FUF}", swap, ""{name=fin, 
middle}, shiftarr={yshift=-10ex}]
&
|[alias=FUFUF]|FUFUF
\arrow[Rightarrow, r, "FU \eps F"]
\arrow[Rightarrow, ur, "\eps FUF"]
&
FUF
\arrow[Rightarrow, r, "\eps F"]
&
F
\arrow[phantom, from=FUFUF, to=fin, 
"\Rrightarrow"{near 
end, rotate=-90}, 
"FU s^{-1}"{near 
start}]
\end{tikzcd}
\end{equation}
are equal by using the identities
\[
\begin{tikzcd}
FUF
\arrow[Rightarrow, r]
\arrow[Rightarrow, d]
\arrow[Rightarrow, dd, shiftarr={xshift=-15ex}, 
""{middle, name=fin}]
&
F
\arrow[Rightarrow, d]
\arrow[dl, phantom, "\Rrightarrow"{rotate=225, 
above}, 
"\eps_{F \eta}"{above}]
\\
|[alias=FUFUF]|FUFUF
\arrow[Rightarrow, r]
\arrow[Rightarrow, d]
&
FUF
\arrow[Rightarrow, d]
\arrow[dl, phantom, "\Rrightarrow"{rotate=225, 
above}, 
"\eps_{\eps F}"{above}]
\\
FUF
\arrow[Rightarrow, r]
&
F
\arrow[phantom, from=FUFUF, to=fin, 
"\Rrightarrow"{above, rotate=225}, "FUs^{-1}"{above}]
\end{tikzcd}
=
\begin{tikzcd}
F
\arrow[Rightarrow, r, "F \eta"]
\arrow[Rightarrow, dr, "1_F", swap, ""{name=fin, 
middle}]
\arrow[dr, bend left, phantom, "s^{-1} \Ldar"{middle}]
&
|[alias=FUF]|FUF
\arrow[Rightarrow, d, "\eps F"] 
\\
&
F
\end{tikzcd}
\]
and
\[
\begin{tikzcd}
F
\arrow[Rightarrow, r, "F \eta"]
\arrow[Rightarrow, d, "F \eta", swap]
&
FUF
\arrow[Rightarrow, r, "\eps F"]
\arrow[Rightarrow, d, "FUF \eta"]
\arrow[dl, phantom, "\Rrightarrow"{rotate=225, 
above}, 
"F \eta_{\eta}"{above}]
&
F
\arrow[Rightarrow, d, "F \eta"]
\arrow[dl, phantom, "\Rrightarrow"{rotate=225, 
above}, 
"\eps_{F \eta}"{above}]
\\
FUF
\arrow[Rightarrow, r, "F \eta UF", swap]
\arrow[Rightarrow, rr, shiftarr={yshift=-10ex}, 
"1_{FUF}", swap, ""{middle, name=fin}]
&
|[alias=FUFUF]|FUFUF
\arrow[Rightarrow, r, "\eps FUF", swap]
&
FUF
\arrow[phantom, from=FUFUF, to=fin, 
"\Rrightarrow"{rotate=-90, near end}, "s^{-1} 
UF"{near 
start}]
\end{tikzcd}
=
\begin{tikzcd}
F
\arrow[Rightarrow, r, "F \eta"]
\arrow[Rightarrow, dr, "1_F", swap, ""{name=fin, 
middle}]
\arrow[dr, bend left, phantom, "s^{-1} \Ldar"{middle}]
&
|[alias=FUF]|FUF
\arrow[Rightarrow, d, "\eps F"] 
\\
&
F
\end{tikzcd}
\]
The proof is therefore finished.
\end{proof}

\begin{lemma}[The implication (1) $\implies$ (2)]
Suppose that $F \adj U: A \to X$ is a 
pseudoadjunction 
in a $\Gray$-category $\gkat{K}$ with
\[
\begin{tikzcd}
X
\arrow[r, "F"]
\arrow[dr, "1", swap, ""{name=beg, middle}]
\arrow[dr, bend left, phantom, "\eta \uRar"{middle}]
&
|[alias=A]|A
\arrow[d, "U", ""{name=fin, middle}]
\\
&
X
\end{tikzcd}
\quad
\begin{tikzcd}
A
\arrow[d, "U", swap]
\arrow[dr, "1_A", ""{name=fin, middle}]
\arrow[dr, bend right, phantom, "\eps \uRar"{middle}]
&
\\
|[alias=X]|X
\arrow[r, "F", swap]
&
A
\end{tikzcd}
\quad
\begin{tikzcd}
F
\arrow[Rightarrow, r, "F \eta"]
\arrow[Rightarrow, dr, "1_F", swap, ""{name=fin, 
middle}]
\arrow[dr, bend left, phantom, "s \Urar"{middle}]
&
|[alias=FUF]|FUF
\arrow[Rightarrow, d, "\eps F"] 
\\
&
F
\end{tikzcd}
\quad
\begin{tikzcd}
U
\arrow[Rightarrow, r, "\eta U"]
\arrow[Rightarrow, dr, "1_U", swap, ""{name=fin, 
middle}]
\arrow[dr, bend left, phantom, "t \Ldar"{middle}]
&
|[alias=UFU]|UFU
\arrow[Rightarrow, d, "U \eps"] 
\\
&
U
\end{tikzcd}
\]
Then
\[
\begin{tikzcd}
X
\arrow[r, "F"]
\arrow[dr, "1", swap, ""{name=beg, middle}]
\arrow[dr, bend left, phantom, "\eta \uRar"{middle}]
&
|[alias=A]|A
\arrow[d, "U", ""{name=fin, middle}]
\\
&
X
\end{tikzcd}
\]
is an absolute left pseudolifting.
\end{lemma}
\begin{proof}
We need to show that for each $G: X \to Y$ and $g: G 
\Rightarrow KF$ there is a 2-cell $g^\sharp: GU 
\Rightarrow K$ and an isomorphism
\[
\begin{tikzcd}
G
\arrow[Rightarrow, r, "G\eta"]
\arrow[Rightarrow, dr, "G", swap, ""{name=fin, 
middle}]
\arrow[dr, bend left, phantom, "\mu(g) \Ldar"{middle}]
&
|[alias=GUF]|GUF
\arrow[Rightarrow, d, "g^\sharp F"] 
\\
&
KF
\end{tikzcd}
\]
satisfying that for each $k: GU \Rightarrow K$ and
\[
\begin{tikzcd}
G
\arrow[Rightarrow, r, "G \eta"]
\arrow[Rightarrow, dr, "g", swap, ""{name=fin, 
middle}]
\arrow[dr, bend left, phantom, "\omega \Ldar"{middle}]
&
|[alias=GUF]|GUF
\arrow[Rightarrow, d, "kF"] 
\\
&
KF
\end{tikzcd}
\]
there is a unique 3-cell $\lift{\omega}: k 
\Rrightarrow g^\sharp$ such that
\[
\begin{tikzcd}
G
\arrow[Rightarrow, r, "G \eta"]
\arrow[Rightarrow, dr, "g", swap, ""{name=fin1, 
middle}]
&
|[alias=GUF]|GUF
\arrow[Rightarrow, d, "g^\sharp F", swap, 
""{name=fin2, middle}] 
\\
&
|[alias=KF]|KF
\arrow[phantom, from=GUF, to=fin1, 
"\Rrightarrow"{near 
end, rotate=225}, 
"\mu(g)"{near 
start}]
\arrow[Rightarrow, from=GUF, to=KF, 
shiftarr={xshift=10ex}, 
"kF", ""{name=beg2, middle}]
\arrow[from=beg2, to=fin2, "\Rrightarrow"{below, 
rotate=180}, 
"\lift{\omega} F"{below}, phantom]
\end{tikzcd}
=
\begin{tikzcd}
G
\arrow[Rightarrow, r, "G \eta"]
\arrow[Rightarrow, dr, "g", swap, ""{name=fin, 
middle}]
\arrow[dr, bend left, phantom, "\omega \Ldar"{middle}]
&
|[alias=GUF]|GUF
\arrow[Rightarrow, d, "kF"] 
\\
&
KF
\end{tikzcd}
\]
We shall define $g^\sharp: GU \Rightarrow K$ as the 
2-cell
\[
\begin{tikzcd}
GU
\arrow[Rightarrow, r, "gU"]
&
KFU
\arrow[Rightarrow, r, "K \eps"]
&
K
\end{tikzcd}
\]
Then, for the 2-cell $g: G \Rightarrow KF$ we define 
$\mu(g)$ to be the 3-cell
\[
\begin{tikzcd}
G
\arrow[Rightarrow, r, "G \eta"]
\arrow[Rightarrow, d, "g", swap]
&
GUF
\arrow[Rightarrow, d, "g UF"]
\arrow[dl, phantom, "\Rrightarrow"{rotate=225, 
above}, 
"{g_{\eta}}^{-1}"{above}]
\\
KF
\arrow[Rightarrow, r, "KF \eta"]
\arrow[Rightarrow, dr, "1_{KF}", swap, ""{name=fin, 
middle}]
&
|[alias=KFUF]|KFUF
\arrow[Rightarrow, d, "K \eps F"] 
\\
&
KF
\arrow[phantom, from=KFUF, to=fin, 
"\Rrightarrow"{near 
end, rotate=225}, 
"Ks^{-1}"{near 
start}]
\end{tikzcd}
\]
Now given a 3-cell
\[
\begin{tikzcd}
G
\arrow[Rightarrow, r, "G \eta"]
\arrow[Rightarrow, dr, "g", swap, ""{name=fin, 
middle}]
\arrow[dr, bend left, phantom, "\omega \Ldar"{middle}]
&
|[alias=GUF]|GUF
\arrow[Rightarrow, d, "kF"] 
\\
&
KF
\end{tikzcd}
\]
we will show that the ``lifted'' 3-cell 
$\lift{\omega}: k \Rrightarrow g^\sharp$
is the 3-cell
\[
\begin{tikzcd}
GU
\arrow[Rightarrow, d, "gU", swap]
\arrow[Rightarrow, r, "1_{GU}"]
&
GU
\arrow[Rightarrow, d, "G \eta U"]
\arrow[Rightarrow, dd, shiftarr={xshift=15ex}, 
""{middle, name=beg}, "1_{GU}"]
\arrow[dl, phantom, "\Rrightarrow"{rotate=135, 
above}, 
"\omega U"{above}]
\\
KFU
\arrow[Rightarrow, d, "K \eps", swap]
&
|[alias=GUFU]|GUFU
\arrow[Rightarrow, l, "kFU"]
\arrow[Rightarrow, d, "GU \eps"]
\arrow[dl, phantom, "\Rrightarrow"{rotate=135, 
above}, 
"{k_\eps}^{-1}"{above}]
\\
K
&
GU
\arrow[Rightarrow, l, "k"]
\arrow[phantom, from=beg, to=GUFU, 
"\Rrightarrow"{rotate=180, above}, "Gt^{-1}"{above}]
\end{tikzcd}
\]
Indeed, observe that the 3-cell
\[
\begin{tikzcd}
G
\arrow[Rightarrow, r, "G \eta"]
\arrow[Rightarrow, dr, "g", swap, ""{name=fin1, 
middle}]
&
|[alias=GUF]|GUF
\arrow[Rightarrow, d, "g^\sharp F", swap, 
""{name=fin2, middle}] 
\\
&
|[alias=KF]|KF
\arrow[phantom, from=GUF, to=fin1, 
"\Rrightarrow"{near 
end, rotate=225}, 
"\mu(g)"{near 
start}]
\arrow[Rightarrow, from=GUF, to=KF, 
shiftarr={xshift=10ex}, 
"kF", ""{name=beg2, middle}]
\arrow[from=beg2, to=fin2, "\Rrightarrow"{below, 
rotate=180}, 
"\lift{\omega} F"{below}, phantom]
\end{tikzcd}
\]
is the composite
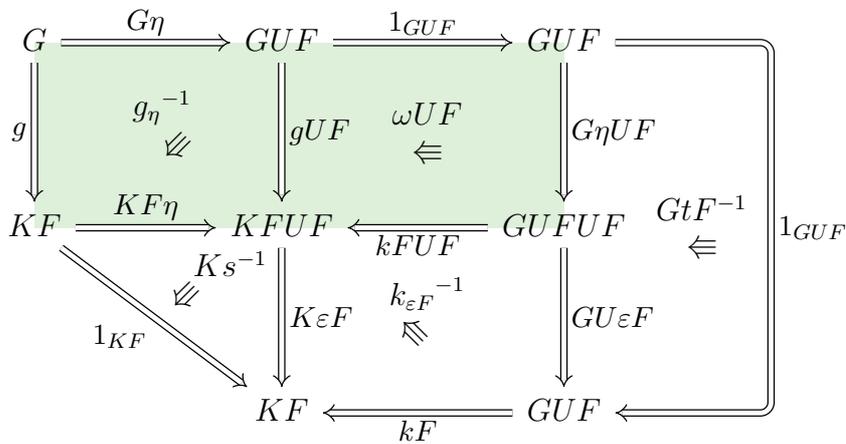
\[
\begin{tikzcd}[
execute at end picture={
  \begin{pgfonlayer}{background}
  \foreach \Nombre in  {n1,n2,GUFUF,n3}
    {\coordinate (\Nombre) at (\Nombre.center);}
  \fill[greeo!20] 
    (n1) -- (n2) -- (GUFUF) -- (n3) -- cycle;
    \end{pgfonlayer}
}
]
|[alias=n1]|G
\arrow[Rightarrow, r, "G \eta"]
\arrow[Rightarrow, d, "g", swap]
&
GUF
\arrow[Rightarrow, d, "g UF"]
\arrow[Rightarrow, r, "1_{GUF}"]
\arrow[dl, phantom, "\Rrightarrow"{rotate=225, 
above}, 
"{g_{\eta}}^{-1}"{above}]
&
|[alias=n2]|GUF
\arrow[Rightarrow, d, "G \eta UF"]
\arrow[Rightarrow, dd, shiftarr={xshift=15ex}, 
""{middle, name=beg}, "1_{GUF}"]
\arrow[dl, phantom, "\Rrightarrow"{rotate=180, 
above}, 
"\omega UF"{above}]
\\
|[alias=n3]|KF
\arrow[Rightarrow, r, "KF \eta"]
\arrow[Rightarrow, dr, "1_{KF}", swap, ""{name=fin, 
middle}]
&
|[alias=KFUF]|KFUF
\arrow[Rightarrow, d, "K \eps F"] 
&
|[alias=GUFUF]|GUFUF
\arrow[Rightarrow, l, "kFUF"]
\arrow[Rightarrow, d, "GU \eps F"]
\arrow[dl, phantom, "\Rrightarrow"{rotate=135, 
above}, 
"{k_{\eps F}}^{-1}"{above}]
\\
&
KF
\arrow[phantom, from=KFUF, to=fin, 
"\Rrightarrow"{near 
end, rotate=225}, 
"Ks^{-1}"{near 
start}]
&
GUF
\arrow[Rightarrow, l, "kF"]
\arrow[phantom, from=beg, to=GUFUF, 
"\Rrightarrow"{rotate=180, above}, "GtF^{-1}"{above}]
\end{tikzcd}
\]
which is equal (transforming the green subdiagram) to 
the 3-cell
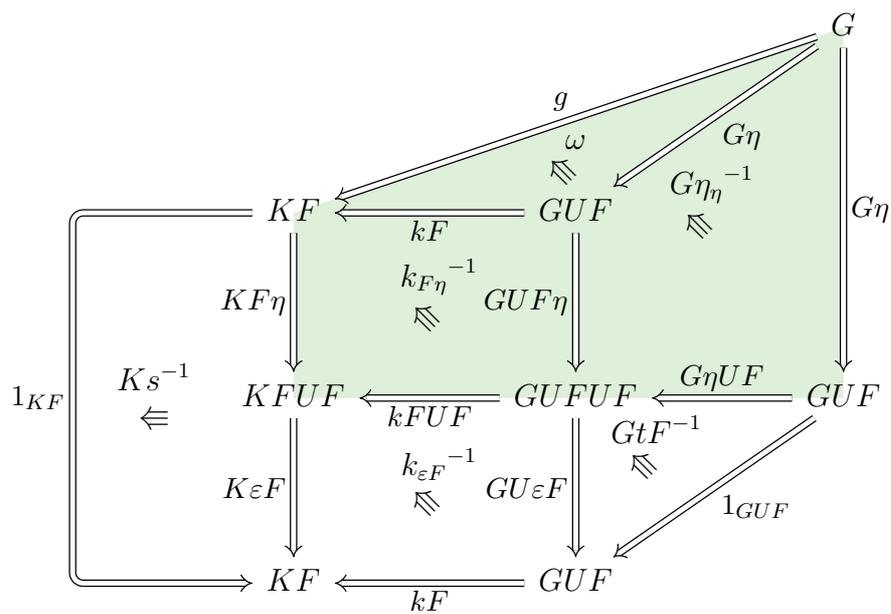
\[
\begin{tikzcd}[
execute at end picture={
  \begin{pgfonlayer}{background}
  \foreach \Nombre in  {n1,n2,n3,n4}
    {\coordinate (\Nombre) at (\Nombre.center);}
  \fill[greeo!20] 
    (n1) -- (n2) -- (n3) -- (n4) -- cycle;
    \end{pgfonlayer}
}
]
&
&
&
|[alias=n1]|G
\arrow[Rightarrow, dd, "G \eta"]
\arrow[Rightarrow, dl, "G \eta"]
\arrow[Rightarrow, dll, "g", swap, ""{middle, 
name=fin1}]
\arrow[ddl, phantom, "\Rrightarrow"{rotate=135, 
above}, 
"{G \eta_{\eta}}^{-1}"{above}]
\\
|[alias=m1]|{}
&
|[alias=n2]|KF
\arrow[Rightarrow, d, "KF \eta", swap]
\arrow[Rightarrow, dd, shiftarr={xshift=-16ex}, 
""{middle, name=fin}, "1_{KF}", swap]
&
|[alias=m2]|GUF
\arrow[Rightarrow, d, "GUF \eta", swap]
\arrow[Rightarrow, l, "kF"]
\arrow[dl, phantom, "\Rrightarrow"{rotate=135, 
above}, 
"{k_{F \eta}}^{-1}"{above}]
\arrow[phantom, to=fin1, "\Rrightarrow"{rotate=135, 
above}, "\omega"{above}]
&
\\
&
|[alias=n3]|KFUF
\arrow[Rightarrow, d, "K \eps F", swap]
\arrow[Rightarrow, phantom, to=fin, 
"\Rrightarrow"{rotate=180, above}, "{Ks}^{-1}"{above}]
&
|[alias=GUFUF]|GUFUF
\arrow[Rightarrow, d, "GU \eps F", swap]
\arrow[Rightarrow, l, "k FUF"]
\arrow[dl, phantom, "\Rrightarrow"{rotate=135, 
above}, 
"{k_{\eps F}}^{-1}"{above}]
&
|[alias=n4]|GUF
\arrow[Rightarrow, l, "G \eta UF", swap]
\arrow[Rightarrow, dl, "1_{GUF}", ""{middle, 
name=beg1}]
\arrow[Rightarrow, phantom, from=beg1, to=GUFUF, 
"\Rrightarrow"{rotate=135, above}, 
"{GtF}^{-1}"{above}]
\\
|[alias=m4]|{}
&
KF
&
|[alias=m3]|GUF
\arrow[Rightarrow, l, "kF"]
&
\end{tikzcd}
\]
and, transforming ${Ks}^{-1}$ to ${GUs}^{-1}$, the 
3-cell
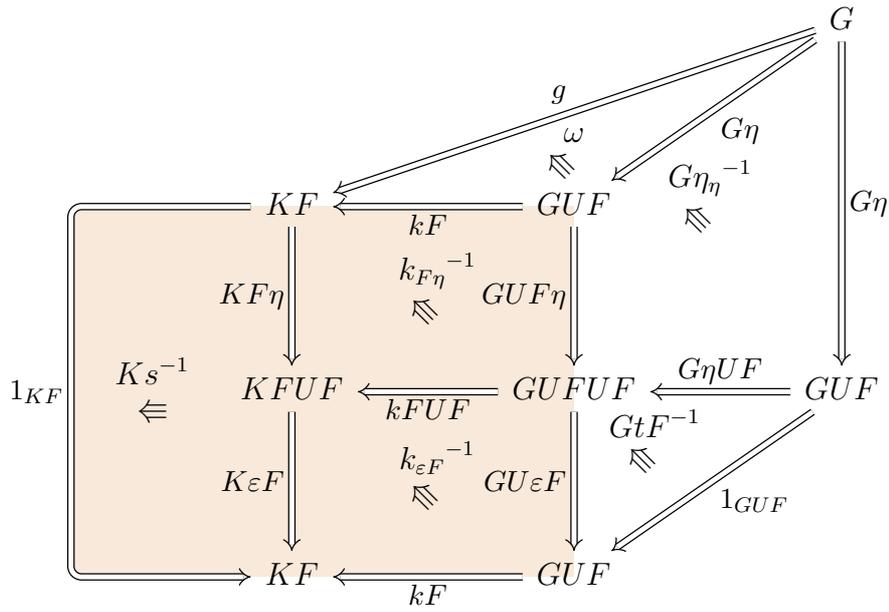
\[
\begin{tikzcd}[
execute at end picture={
  \begin{pgfonlayer}{background}
    \foreach \Nombre in  {m1,m2,m3,m4}
      {\coordinate (\Nombre) at (\Nombre.center);}
    \fill[cof!20] 
      (m1) -- (m2) -- (m3) -- (m4) -- cycle;
    \end{pgfonlayer}
}
]
&
&
&
|[alias=n1]|G
\arrow[Rightarrow, dd, "G \eta"]
\arrow[Rightarrow, dl, "G \eta"]
\arrow[Rightarrow, dll, "g", swap, ""{middle, 
name=fin1}]
\arrow[ddl, phantom, "\Rrightarrow"{rotate=135, 
above}, 
"{G \eta_{\eta}}^{-1}"{above}]
\\
|[alias=m1]|{}
&
|[alias=n2]|KF
\arrow[Rightarrow, d, "KF \eta", swap]
\arrow[Rightarrow, dd, shiftarr={xshift=-16ex}, 
""{middle, name=fin}, "1_{KF}", swap]
&
|[alias=m2]|GUF
\arrow[Rightarrow, d, "GUF \eta", swap]
\arrow[Rightarrow, l, "kF"]
\arrow[dl, phantom, "\Rrightarrow"{rotate=135, 
above}, 
"{k_{F \eta}}^{-1}"{above}]
\arrow[phantom, to=fin1, "\Rrightarrow"{rotate=135, 
above}, "\omega"{above}]
&
\\
&
|[alias=n3]|KFUF
\arrow[Rightarrow, d, "K \eps F", swap]
\arrow[Rightarrow, phantom, to=fin, 
"\Rrightarrow"{rotate=180, above}, "{Ks}^{-1}"{above}]
&
|[alias=GUFUF]|GUFUF
\arrow[Rightarrow, d, "GU \eps F", swap]
\arrow[Rightarrow, l, "k FUF"]
\arrow[dl, phantom, "\Rrightarrow"{rotate=135, 
above}, 
"{k_{\eps F}}^{-1}"{above}]
&
|[alias=n4]|GUF
\arrow[Rightarrow, l, "G \eta UF", swap]
\arrow[Rightarrow, dl, "1_{GUF}", ""{middle, 
name=beg1}]
\arrow[Rightarrow, phantom, from=beg1, to=GUFUF, 
"\Rrightarrow"{rotate=135, above}, 
"{GtF}^{-1}"{above}]
\\
|[alias=m4]|{}
&
KF
&
|[alias=m3]|GUF
\arrow[Rightarrow, l, "kF"]
&
\end{tikzcd}
\]
is equal to
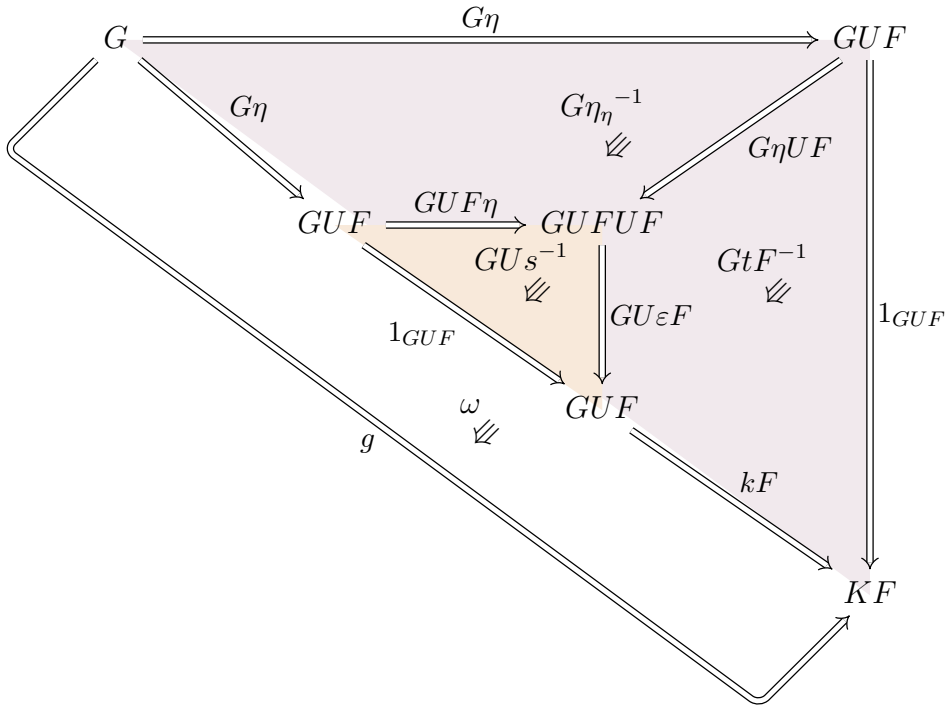
\[
\begin{tikzcd}[arrows=Rightarrow,
execute at end picture={
  \begin{pgfonlayer}{background}
  \foreach \Nombre in  {m1,m2,m3}
    {\coordinate (\Nombre) at (\Nombre.center);}
  \fill[pur!20] 
    (m1) -- (m2) -- (m3)  -- cycle;  
  \foreach \Nombre in  {n1,n2,GUFUF}
    {\coordinate (\Nombre) at (\Nombre.center);}
  \fill[cof!20] 
    (n1) -- (n2) -- (GUFUF)  -- cycle;
    \end{pgfonlayer}
}
]
|[alias=m1]|G
\arrow[rrr, "G \eta"]
\arrow[dr, "G \eta"]
\arrow[dddrrr, shiftarr={xshift=-8ex,yshift=-8ex}, 
"g", swap, ""{middle, name=fin2}]
&
&
&
|[alias=m2]|GUF
\arrow[dl, "G \eta UF"]
\arrow[ddd, "1_{GUF}", ""{middle, name=beg1}]
\arrow[dll, phantom, "\Rrightarrow"{rotate=225, 
above}, "{G\eta_{\eta}}^{-1}"{above}]
\arrow[from=beg1, to=GUFUF, phantom, 
"\Rrightarrow"{rotate=225, above}, 
"{GtF}^{-1}"{above}]
\\
&
|[alias=n1]|GUF
\arrow[r, "GUF \eta"]
\arrow[dr, "1_{GUF}", swap, ""{middle, name=fin1}]
&
|[alias=GUFUF]|GUFUF
\arrow[d, "GU \eps F"]
\arrow[to=fin1, phantom, "\Rrightarrow"{rotate=225, 
above}, 
"{GUs}^{-1}"{above}]
&
\\
&
&
|[alias=n2]|GUF
\arrow[rd, "kF"]
\arrow[to=fin2, phantom, "\Rrightarrow"{rotate=225, 
above}, 
"\omega"{above}]
&
\\
&
&
&
|[alias=m3]|KF
\end{tikzcd}
\]
And since the whole coloured subdiagram equals 
identity by 
the pseudomonad coherence axiom, the diagram 
simplifies to 
$\omega$, showing that our choice of $\lift{\omega}$ 
was 
correct. Indeed, the choice of $\lift{\omega}$ is 
even the 
only possible one: each diagram in the following 
series is 
equal to $\lift{\omega}$.
\begin{equation}
\begin{tikzcd}[arrows=Rightarrow]
GU
\arrow[r, bend left, "k", ""{middle, name=beg}]
\arrow[r, bend right, "g^\sharp", swap, ""{middle, 
name=fin}]
\arrow[r, phantom, "\lift{\omega} \Ddar"{middle}]
&
K
\end{tikzcd}
\end{equation}

\begin{equation}
\begin{tikzcd}[arrows=Rightarrow]
&
|[alias=GUFU]|GUFU
\arrow[dr, "GU \eps"]
\arrow[to=fin2, phantom, "\Rrightarrow"{rotate=-90}, 
"Gt"{near start}]
&
&
\\
GU
\arrow[ur, "G \eta U"]
\arrow[rr, "1_{GU}", swap, ""{middle, name=fin2}]
\arrow[d, "gU", swap]
\arrow[rr, shiftarr={yshift=24ex}, "1_{GU}", 
""{middle, name=beg1}]
&
&
|[alias=GU]|GU
\arrow[r, "k"]
\arrow[d, "gU", swap]
\arrow[dll, phantom, "\Rrightarrow"{rotate=225, 
above}, "1_{gU}"{above}]
&
K
\\
KFU
\arrow[rr, "1_{KFU}", swap]
&
&
KFU
\arrow[ur, "K \eps", swap, ""{middle, name=fin1}]
&
\arrow[from=GU, to=fin1, phantom, 
"\Rrightarrow"{rotate=-45, below}, 
"\lift{\omega}"{above}]
\arrow[from=beg1, to=GUFU, phantom, 
"\Rrightarrow"{rotate=-90, near end}, 
"{Gt}^{-1}"{near 
start}]
\end{tikzcd}
\end{equation}

\begin{equation}
\begin{tikzcd}[arrows=Rightarrow]
&
|[alias=GUFU]|GUFU
\arrow[dr, "GU \eps"]
\arrow[d, "gUFU"]
&
&
\\
GU
\arrow[ur, "G \eta U"]
\arrow[d, "gU", swap]
\arrow[rr, shiftarr={yshift=24ex}, "1_{GU}", 
""{middle, name=beg1}]
\arrow[r, phantom, "\Rrightarrow"{near start, 
rotate=225}, 
"{g_{\eta U}}^{-1}"{near end}]
&
|[alias=KFUFU]|KFUFU
\arrow[rd, "KFU \eps"]
\arrow[r, phantom, "\Rrightarrow"{near start, 
rotate=225}, 
"{g_{U \eps}}^{-1}"{near end}]
&
|[alias=GU]|GU
\arrow[r, "k"]
\arrow[d, "gU", swap]
&
K
\\
KFU
\arrow[rr, "1_{KFU}", swap, ""{middle, name=fin2}]
\arrow[ur, "KF \eta U"]
&
&
KFU
\arrow[ur, "K \eps", swap, ""{middle, name=fin1}]
&
\arrow[from=GU, to=fin1, phantom, 
"\Rrightarrow"{rotate=-45, below}, 
"\lift{\omega}"{above}]
\arrow[from=beg1, to=GUFU, phantom, 
"\Rrightarrow"{rotate=-90, near end}, 
"{Gt}^{-1}"{near 
start}]
\arrow[from=KFUFU, to=fin2, phantom, 
"\Rrightarrow"{rotate=-90}, 
"KFt"{near 
start}]
\end{tikzcd}
\end{equation}

\begin{equation}
\begin{tikzcd}[arrows=Rightarrow]
&
|[alias=GUFU]|GUFU
\arrow[dr, "kFU", ""{middle, name=beg2}]
\arrow[d, "gUFU", swap]
\arrow[r, "GU \eps"]
&
GU
\arrow[dr, "k"]
\arrow[d, phantom, "\Rrightarrow"{rotate=-90}, 
"{k_\eps}^{-1}"{near start}]
&
\\
GU
\arrow[ur, "G \eta U"]
\arrow[d, "gU", swap]
\arrow[rru, shiftarr={yshift=16ex}, "1_{GU}", 
""{middle, name=beg1}]
\arrow[r, phantom, "\Rrightarrow"{near start, 
rotate=225}, "{g_{\eta U}}^{-1}"{near end}]
&
|[alias=KFUFU]|KFUFU
\arrow[rd, "KFU \eps", swap]
\arrow[r, "K \eps FU", swap]
\arrow[from=beg2, phantom, "\Rrightarrow"{rotate=225, 
above}, 
"\lift{\omega} FU"{above}]
&
|[alias=KFU]|KFU
\arrow[r, "K \eps", swap]
\arrow[d, phantom, "\Rrightarrow"{rotate=-90}, 
"K\eps_\eps"{near start}]
&
K
\\
KFU
\arrow[rr, "1_{KFU}", swap, ""{middle, name=fin2}]
\arrow[ur, "KF \eta U"]
&
&
KFU
\arrow[ur, "K \eps", swap, ""{middle, name=fin1}]
&
\arrow[from=beg1, to=GUFU, phantom, 
"\Rrightarrow"{rotate=-90, near end}, 
"{Gt}^{-1}"{near 
start}]
\arrow[from=KFUFU, to=fin2, phantom, 
"\Rrightarrow"{rotate=-90}, 
"KFt"{near 
start}]
\end{tikzcd}
\end{equation}

\begin{equation}
\begin{tikzcd}[arrows=Rightarrow]
&
|[alias=GUFU]|GUFU
\arrow[dr, "kFU", ""{middle, name=beg2}]
\arrow[d, "gUFU", swap]
\arrow[r, "GU \eps"]
&
GU
\arrow[dr, "k"]
\arrow[d, phantom, "\Rrightarrow"{rotate=-90}, 
"{k_\eps}^{-1}"{near start}]
&
\\
GU
\arrow[ur, "G \eta U"]
\arrow[d, "gU", swap]
\arrow[rru, shiftarr={yshift=16ex}, "1_{GU}", 
""{middle, name=beg1}]
\arrow[r, phantom, "\Rrightarrow"{near start, 
rotate=225}, "{g_{\eta U}}^{-1}"{near end}]
&
|[alias=KFUFU]|KFUFU
\arrow[r, "K \eps FU", swap]
\arrow[from=beg2, phantom, "\Rrightarrow"{rotate=225, 
above}, 
"\lift{\omega} FU"{above}]
&
|[alias=KFU]|KFU
\arrow[r, "K \eps", swap]
&
K
\\
KFU
\arrow[ur, "KF \eta U"]
\arrow[rrru, bend right, "1_{KFU}", swap, ""{middle, 
name=fin1}]
&
&
&
\arrow[from=beg1, to=GUFU, phantom, 
"\Rrightarrow"{rotate=-90, near end}, 
"{Gt}^{-1}"{near 
start}]
\arrow[from=KFUFU, to=fin1, phantom, 
"\Rrightarrow"{rotate=-90, near end}, 
"{KsU}^{-1}"]
\end{tikzcd}
\end{equation}

\begin{equation}
\begin{tikzcd}
GU
\arrow[Rightarrow, d, "gU", swap]
\arrow[Rightarrow, r, "1_{GU}"]
&
GU
\arrow[Rightarrow, d, "G \eta U"]
\arrow[Rightarrow, dd, shiftarr={xshift=15ex}, 
""{middle, name=beg}, "1_{GU}"]
\arrow[dl, phantom, "\Rrightarrow"{rotate=180, 
above}, 
"\omega U"{above}]
\\
KFU
\arrow[Rightarrow, d, "K \eps", swap]
&
|[alias=GUFU]|GUFU
\arrow[Rightarrow, l, "kFU"]
\arrow[Rightarrow, d, "GU \eps"]
\arrow[dl, phantom, "\Rrightarrow"{rotate=135, 
above}, 
"{k_\eps}^{-1}"{above}]
\\
K
&
GU
\arrow[Rightarrow, l, "k"]
\arrow[phantom, from=beg, to=GUFU, 
"\Rrightarrow"{rotate=180, above}, "Gt^{-1}"{above}]
\end{tikzcd}
\end{equation}
The proof is therefore complete.
\end{proof}

\begin{remark}
Having proved the implications (1) $\implies$ (2) and (3) $\implies$ (1), the proof of Theorem~\ref{thm:pseudo-benabou} is complete.
\end{remark}

%
%

\end{document}